\documentclass[11pt,letterpaper]{amsart}
\usepackage{fancyhdr}
\usepackage{bm}
\usepackage{hyperref}
\usepackage{graphicx} 
\usepackage{nicematrix}
\usepackage{amsthm,amssymb,amsmath}
\usepackage[T1]{fontenc}
\usepackage[utf8]{inputenc}
\usepackage{hyperref}
\usepackage{lipsum}
\usepackage{mathrsfs}
\usepackage{enumitem}
\usepackage{stmaryrd}
\usepackage{setspace}
\usepackage{yfonts}
\usepackage{float}
\usepackage{mathtools}
\usepackage{parskip}
\usepackage[all,cmtip]{xy}
\usepackage{tikz-cd}
\tikzcdset{row sep/normal=50pt, column sep/normal=50pt}

\newtheorem{lemma}{Lemma}[section]
\newtheorem{remark}[lemma]{Remark}
\newtheorem{theorem}[lemma]{Theorem}
\newtheorem{theorem*}{Theorem}

\newtheorem{example*}[lemma]{Example}

\newtheorem{corollary}[lemma]{Corollary}

\newtheorem{manualtheoreminner}{Theorem}
\newenvironment{manualtheorem}[1]{%
  \renewcommand\themanualtheoreminner{#1}%
  \manualtheoreminner
}{\endmanualtheoreminner}
\newtheorem{manualquestioninner}{Question}
\newenvironment{manualquestion}[1]{%
  \renewcommand\themanualquestioninner{#1}%
  \manualquestioninner
}{\endmanualquestioninner}

\usepackage{blindtext}

\usepackage[backend=biber, style=numeric, sorting=nyt]{biblatex}
\title{Tangent bundle of punctual Hilbert scheme of a surface, and a conjecture of Dey-Mukherjee-Pahari}
\author{\small{Supravat Sarkar}}
\date{}
\begin{document}

\begin{abstract}
We describe the indecomposable components of the tangent bundle of the punctual Hilbert scheme of a smooth projective surface. As an application, we prove a recent conjecture about classification of products of punctual Hilbert schemes of smooth projective surfaces. We also determine when two products of symmetric powers of a smooth variety can be isomorphic. As a key step in our proof, we give a new characterization of abelian varieties, which states that in dimension $\geq 2$, a smooth complex projective variety whose tangent bundle is trivial upto a line bundle twist is an abelian variety.
\end{abstract}
\maketitle
\begin{center}
\end{center}
\begin{center}
\end{center}

\section{Introduction}

We work throughout over the field $\mathbb{C}$ of complex numbers. Algebraic varieties will always be assumed to be integral. For a smooth projective surface $X$ and a positive integer $n$, the punctual Hilbert scheme $X^{[n]}$ parametrizing length $n$ subschemes of $X$ is a widely studied object. \cite{fogarty1968algebraic} showed $X^{[n]}$ is a smooth projective variety of dimension $2n$. For $X$ a $K3$ surface, $X^{[n]}$ becomes a hyperkähler manifold. For $X$ an abelian surface, one gets a hyperkähler manifold $Kum_n(X)$ closely related to $X^{[n]}$.

Several properties of $X^{[n]}$ are studied throughout the literature. Automorphisms of $X^{[n]}$ are studied by several authors, for example \cite{belmans2020automorphisms},\cite{Og},\cite{Ha},\cite{Wa}, \cite{bansal2025automorphisms}. Another direction of exploring $X^{[n]}$ is to study vector bundles on $X^{[n]}$. This has been pursued for the so-called tautological bundles in \cite{wandel2013stability}, \cite{oprea2022big},\cite{scala2015higher}, \cite{stapleton2016geometry} among others.

For any smooth projective variety $Y$, a very natural vector bundle on $Y$ is the tangent bundle, which we will always denote by $T_Y$. Study of $T_Y$ often gives important geometric information about $Y$, as can be seen, for example, from the famous paper \cite{mori1979projective} of Mori. Tangent bundles of varieties has been studied from several viewpoints, for example in \cite{mehta1987varieties}, \cite{liu2023moment},\cite{tian1992stability} and \cite{conde2004manifolds}.

One of the goals of this article is to study tangent bundle of $X^{[n]}$ for a smooth projective surface $X$. By \cite[Theorem 3]{atiyah1956krull}, every vector bundle on a projective variety can be written as a direct sum of indecomposable vector bundles, which are uniquely determined upto isomorphism. Here a vector bundle is called \textit{indecomposable} if it is not a direct sum of two nonzero vector bundles. Our result describes these indecomposable components of $T_{X^{[n]}}.$ The precise result is Theorem \ref{main}, we state here the following weaker version which is easier to state.
\begin{manualtheorem}{A}\label{A}
    Let $X$ be a smooth projective surface, $n \geq 2$ an integer. Then $T_{X^{[n]}}$ is indecomposable, except possibly in the following cases:
    \begin{enumerate}
        \item $X$ is an abelian surface,
        \item $X$ is minimal and has an isotrivial elliptic fibration,
        \item $X$ is a $\mathbb{P}^1$-bundle over an elliptic curve.
    \end{enumerate}

\end{manualtheorem}
The idea of the proof of Theorem \ref{A} is as follows. Let $S_n$ be the symmetric group on $n$ letters, acting on $X^n$ by permuting the factors. Given a direct sum decomposition of $T_{X^{[n]}}$, we can obtain a direct sum decomposition of $T_{X^n}$ by $S_n$-invariant subbundles, as we explain in the beginning of \S 5. In \S 4, we show such direct sum decomposition of $T_{X^n}$ has very limited possiblities. Finally, in \S 5, we show most of those possibilities do not arise from a direct sum decomposition of $T_{X^{[n]}}$. Along the way, we need to characterize smooth projective surfaces with some particular splitting of tangent bundle. This we do in Theorem \ref{L+L} and Theorem \ref{O+K} in \S 3, which are interesting in their own right. Though we need Theorem \ref{L+L} only for surfaces, we prove it in any dimension. It gives a new characterization of abelian varieties, generalizing the well-known result that the only smooth complex projective varieties with trivial tangent bundle are abelian varieties.

Next, we turn to the following question:
\begin{manualquestion}{1}\label{q}
Suppose $X_1,\cdots ,X_r, Y_1,\cdots ,Y_s$ are varieties such that $$\prod_i X_i\cong \prod _j Y_j.$$ Can we conclude that $r=s$ and upto renumbering $X_i\cong Y_i$ for all $i$?
\end{manualquestion}

Of course without any assumption on the varieties the answer to the question is negative, as the varieties themselves can be products. Even when the varieties are smooth projective and indecomposable, that is, not a nontrivial product, the answer can be negative, as one can see by choosing the varieties to be suitable abelian varieties as in \cite[Lemma 3]{poonen2002grothendieck}. In general when there is no obvious counterexample, this question is very subtle, especially when the varieties are not proper. For example, the famous Zariski cancellation problem is an instance of this question, see \cite{gupta2015survey}. In the case the varieties are projective, indecomposable with irregularity $0$, the question has an affirmative answer, essentially by the same proof as of \cite[Theorem 2.7, Corollary 2.8]{Ha}. When the varieties are smooth projective indecomposable with ample canonical class, the question has an affirmative answer due to \cite[Theorem 4.2]{BHPS}.

As an application of Theorem \ref{main}, we give an affirmative answer to Question \ref{q} when the varieties are punctual Hilbert schemes of a fixed smooth projective surface. This proves a recent conjecture in \cite{dey2026classification} about classification of products of Hilbert schemes, and extends the corresponding result for Hilbert scheme of points on a smooth projective curve as proved in \cite{mukherjee2025diagonal}.
\begin{manualtheorem}{B}\label{conjecture}

Let $X$ be a smooth projective surface, and $a_1, \dots, a_r$, $b_1, \dots, b_s$ positive integers.
Suppose
\[
\prod_i X^{[a_i]} \cong \prod_j X^{[b_j]}.
\]
Then $r=s$ and
$ \{a_1, \dots, a_r\} = \{b_1, \dots, b_s\}
$
as multisets.
\end{manualtheorem}

Next, we give an affirmative answer to question \ref{q} for symmetric powers of a (not necessarily proper) smooth variety $X$ of any dimension, denoted by $X^{(a)}$.
\begin{manualtheorem}{C}\label{C}
   Let $X$ be a smooth variety of dimension $\geq 2$, and suppose
\[
\prod_i X^{(a_i)} \cong \prod_j X^{(b_j)}.
\]
Then $r = s$ and $\{a_1,\dots,a_r\} = \{b_1,\dots,b_s\}$
as multisets.
\end{manualtheorem}

We prove Theorem \ref{main} in \S 6, and Theorem \ref{conjecture} and \ref{C} in \S 7.

\section{Notations and conventions}
\begin{itemize}
    \item For complex manifolds $X_1,\dots,X_n$, with $Y=\prod_iX_i$ and projections $\pi_i:Y\to X_i$, we will always identify $\pi_i^*T_{X_i}$ as a direct summand subbundle of $T_Y$ via the natural isomorphism $T_Y\cong\oplus_i\pi_i^*T_{X_i}.$
    \item For a group $G$ acting on a complex manifold $Y$, we say the action is a \textit{covering space action} if $G$ is the group of deck transformations of a holomorpic covering space $Y\to X.$
    \item For a variety $X$, we denote the singular locus of $X$ by $\mathrm{Sing }(X).$
    \item For a positive integer $n$, $\mathbb{D}^n$ is the polydisc $$\{(z_1,\cdots,z_n)\in \mathbb{C}^n\big| |z_i|<1 \text{ for all } i\}.$$ We abbreviate $\mathbb{D}^1$ as $\mathbb{D}.$
    \item For a smooth projective variety $X$, the canonical bundle of $X$ will be denoted by $K_X.$
    \item Let $\mathcal{C}$ be the class of smooth projective surfaces $X$ which are of the form $(D\,\times\, Y)/G$, where $D$ is an elliptic curve, $Y$ a smooth curve of general type, and $G$ is a finite subgroup of $\text{Aut}^0(D)$ acting diagonally on $D\times Y.$ This is what was called class $E$ in {\cite[Theorem D]{fong2024connected}}. In this case, $X$ is minimal of Kodaira dimension $1$, and the Iitaka fibration of $X$ is an isotrivial elliptic fibration. 
    \item Call a smooth projective surface $X$ to be of class $\tilde{\mathcal{C}}$ if $X$ is either of class $\mathcal{C}$, or a bielliptic surface, or a $\mathbb{P}^1$-bundle over an elliptic curve with $T_X$ split, that is, direct sum of two line bundles. These $\mathbb{P}^1$-bundles over an elliptic curve $E$ are completely classified in \cite{ganong1985tangent}: these are precisely $\mathbb{P}_E(\mathcal{E})$ for an indecomposable rank $2$ vector bundle $\mathcal{E}$ on $E$ (by \cite{atiyah1957vector} there are only $2$ such ruled surfaces), and $\mathbb{P}_E(\mathcal{O}_E\oplus L)$ for a degree $0$ line bundle $L$ on $E.$
\end{itemize}
\section{Splitting of tangent bundle}
The goal of this section is to prove Theorems \ref{L+L} and \ref{O+K}.
\begin{theorem}\label{L+L}
    Let $X$ be a smooth projective variety of dimension $n\geq 2$ whose tangent bundle is trivial upto a line bundle twist. Then $X$ is an abelian variety.
\end{theorem}
\begin{proof}
For a vector bundle $E$ on a smooth projective variety $Y$, let
\[
c_i(E)\in H^i(Y,\Omega_Y^i)\subset H^{2i}(Y,\mathbb{C})
\]
be the $i$'th Chern class of $E$. Let
\[
c(E)=\sum_{i\ge0}c_i(E)\in H^*(Y,\mathbb{C})
\]
be the total Chern class.

Let $T_X=\bigoplus_{i=1}^{n}L_i,$
where the $L_i$'s are pairwise isomorphic line subbundles of $T_X$. So, $H^1(X,\Omega_X)
=
\bigoplus_{i=1}^{n}
H^1(X,L_i^{-1}),$
and by \cite[Lemma 3.1]{beauville2000complex}, $c_1(L_i)\in H^1(X,\Omega_X)$
lies in the direct summand $H^1(X,L_i^{-1})$. Here we use the standard fact that the Atiyah class of the line bundle $L_i$ is a nonzero scalar multiple of $c_1(L_i)$; see \cite{atiyah1957complex}. As the $L_i$'s are all isomorphic, all $c_1(L_i)$'s are equal. Since $n\ge2$, we obtain $c_1(L_i)=0$
for all $i$. Therefore,
\[
c(T_X)
=
\prod_{i}(1+c_1(L_i))
=
1,
\]
that is, $c_j(T_X)=0$
for all $j>0$.

By \cite[Theorem 4.1]{catanese2022manifolds},
$X=A/G$, where $A$ is an abelian variety and $G$ is a finite group acting freely on $A$ without translations.

Let $\pi:A\to X$
be the quotient map. The line bundles $\pi^*M_i$ are all trivial by \cite{atiyah1956krull}, as
\[
\bigoplus_i \pi^*M_i
\cong
\pi^*T_A
\cong
T_X
\cong
\mathcal{O}_A^{\oplus n}.
\]
Also, the $\pi^*M_i$'s are all $G$-equivariantly isomorphic, as the $L_i$'s are all isomorphic. Hence there is a character $\chi:G\longrightarrow\mathbb{C}^*$
such that each $\pi^*M_i$ is $G$-equivariantly isomorphic to $\mathcal{O}_A$ with linearization given by $\chi$. Consequently, $T_A$
is $G$-equivariantly isomorphic to $\mathcal{O}_A^{\oplus n}$ with $G$-linearization given by the map
\[
G
\xrightarrow{\ \eta\ }
GL(n,\mathbb{C}),
\qquad
g\longmapsto \chi(g)I.
\]

In other words, after identifying $T_0A=\mathbb{C}^n$, under the natural trivialization of $T_A$ by translations, the $G$-linearization of $T_A$ is exactly the one given by $\eta$.

Now we show that $G=1$. Suppose not. Let $1\neq g\in G$.
There is $a\in A$ and a group automorphism $\varphi$ of $A$ such that
\[
g\cdot x=\varphi(x)+a
\]
for all $x\in A$. We have $(d\varphi)_0=\eta(g)$ in $GL(n,\mathbb{C})$.
As $G$ acts without translations, we have $\varphi\neq\mathrm{id}$, so $(d\varphi)_0\neq I.$
Hence $\chi(g)\neq1.$
Therefore,
\[
d(1-\varphi)_0=(1-\chi(g))I
\]
is invertible, where $1$ denotes the identity automorphism of $A$. Thus the group homomorphism $1-\varphi:A\to A$
is an isogeny. Hence there is $x_0\in A$ such that $(1-\varphi)(x_0)=a.$
Therefore, $g\cdot x_0=x_0,$
a contradiction as $G$ acts freely on $A$. So $G=1$, hence $X=A$
is an abelian variety.
\end{proof}

\begin{theorem}\label{O+K}
    Let $X$ be a smooth projective surface whose tangent bundle has a trivial direct summand.
Then $X$ is an abelian surface or of class $\Tilde{\mathcal{C}}$.
\end{theorem}
\begin{proof}
    We have
\[
T_X \cong \mathcal{O}_X \oplus L
\] for a line bundle $L$ on $X$.
Taking determinants, we get
\[
L \cong K_X^{-1}.
\]
So,
\[
\Omega_X \cong \mathcal{O}_X \oplus K_X.
\]

As $\Omega_X$ has a nowhere vanishing section, by \cite{popa2014kodaira} we have $\kappa(X) \leq 1$.

Suppose $\kappa(X) \neq -\infty$. Note that
\[
\dim \operatorname{Aut}^0(X) = h^0(X,T_X) > 0,
\] where $\operatorname{Aut}^0(X)$ is the identity component of the automorphism group scheme of $X.$
By \cite{fong2024connected}, 
$X$ is of class $\mathcal{C}$ or an abelian surface or a bielliptic surface.

Suppose $\kappa(X) = -\infty$. So, $h^0(X,K_X) = 0.$
Hence
\[
h^0(X,\Omega_X) = h^0(X,\mathcal{O}_X) = 1.
\] By \cite[Theorem C]{beauville2000complex}, $X$ has no $(-1)$-curve. So $X$ is a $\mathbb{P}^1$-bundle over an elliptic curve. As $T_X$ splits, $X$ is in class $\Tilde{\mathcal{C}}$.

\end{proof}
\section{$S_n$-invariant subbundles of $T_{X^n}$ }
The goal of this section is to prove Theorems \ref{Sn invariant1} and Theorem \ref{Sn invariant}. For a smooth projective surface $X$ and an integer $n\geq 2$, the action of $S_n$ on $X^n$ induces an action of $S_n$ on $T_{X^n}$, that is, an $S_n$-linearization of $T_{X^n}$. A subbundle of $T_{X^n}$ is called $S_n$-invariant if it is invariant under this action.
\begin{theorem}\label{Sn invariant1}
    Let $X$ be a smooth projective surface with $T_X$ indecomposable, and $n \geq 2$ an integer. Then $T_{X^n}$ can not be written as a direct sum of two nonzero $S_n$-invariant subbundles.
\end{theorem}
\begin{proof}
Let $\pi_i:X^n\to X$ be the projections. Suppose we can write $T_{X^n} = \mathcal{E}_1 \oplus \mathcal{E}_2$, with $\mathcal{E}_1, \mathcal{E}_2$ nonzero $S_n$-invariant subbundles of $T_{X^n}$. Note that $\pi_i^* T_X$ is indecomposable, as its restriction to a section of $\pi_i$ is $T_X$, which is indecomposable. As $T_X$ is nontrivial, $\pi_i^* T_X$'s are pairwise non-isomorphic.

Since
\[
\bigoplus_i \pi_i^* T_X \cong \mathcal{E}_1 \oplus \mathcal{E}_2,
\]
by \cite[Theorem 3]{atiyah1956krull}, there exists $I \subset[n]$ such that we have
\[
\mathcal{E}_1= \bigoplus_{i\in I} \pi_i^* T_X,
\]
and 
\[\mathcal{E}_2 = \bigoplus_{i\not\in I} \pi_i^* T_X.
\] As $\mathcal{E}_1$ is $S_n$-invariant, we have 
\[
\bigoplus_{i \in I} \pi_i^* T_X \cong \bigoplus_{i \in \sigma(I)} \pi_i^* T_X
\quad \text{for all } \sigma \in S_n.
\]

As $\pi_i^* T_X$'s are pairwise non-isomorphic, by \cite[Theorem 3]{atiyah1956krull} we have $I = \sigma(I)$ for all $\sigma \in S_n$. Hence $I = \emptyset$ or $[n]$. This is a contradiction since $\mathcal{E}_1$ and $\mathcal{E}_2$ are nonzero.

\end{proof}
\begin{theorem}\label{Sn invariant}
    Let $X$ be a smooth projective surface that is not an abelian surface, and $n \geq 2$ an integer. Suppose there are line subbundles $L, M$ of $T_X$ with $T_X = L \oplus M.$ Suppose $\mathcal{E}_1, \mathcal{E}_2$ are $S_n$-invariant subbundles of $T_{X^n}$ with $0<\textrm{rank }\mathcal{E}_1\leq \textrm{rank }\mathcal{E}_2$ and
\[
T_{X^n} = \mathcal{E}_1 \oplus \mathcal{E}_2.
\]
Let $\pi_i:X^n\to X$ be the projections.
Then, after swapping $L, M$ if necessary, one of the following holds:

\begin{enumerate}
\item[(a)]
$\mathcal{E}_2 = \bigoplus_{i=1}^n \pi_i^* M.$
\item[(b)]
$L$ is trivial, and denoting
\[
\mathcal{F}_1 = \operatorname{im} \left( \mathcal{O}_{X^n} \xlongrightarrow{(\pi_i^*\xi)_i} \bigoplus_i \pi_i^* L \right),
\]
and
\[
\mathcal{F}_2 = \ker \left( \bigoplus_i \pi_i^* L \xlongrightarrow{(\pi_i^*\xi^{-1})_i} \mathcal{O}_{X^n} \right),
\] where  $\xi : \mathcal{O}_X \longrightarrow L$ is a trivialization, we have

\begin{itemize}
\item[ either (i)] $\mathcal{E}_1= \mathcal{F}_k$ for some $k = 1,2$, and $h^0(X, M) = 0$,
\item[or (ii)] $\mathcal{E}_2 = \mathcal{F}_k \oplus \left( \bigoplus_i \pi_i^* M \right)$, for $k = 1,2$, and $h^0(X, M) \neq 0$.
\end{itemize}
\end{enumerate}
\end{theorem}
\begin{proof}
  Note that  \[
T_{X^n} = \left( \bigoplus_i \pi_i^* L \right) \oplus \left( \bigoplus_j \pi_j^* M \right)
\]
is a decomposition into line bundles. By \cite[Theorem 3]{atiyah1956krull}, there are $I, J \subseteq [n]$ with
\[
\mathcal{E}_1 \cong \left( \bigoplus_{i \in I} \pi_i^* L \right) \oplus \left( \bigoplus_{j \in J} \pi_j^* M \right),
\]
and 
\[
\mathcal{E}_2 \cong \left( \bigoplus_{i \not\in I} \pi_i^* L \right) \oplus \left( \bigoplus_{j \not\in J} \pi_j^* M \right).
\]

Since $\mathcal{E}_1$ is $S_n$-invariant, we have
\begin{equation}\label{LM}
   \left( \bigoplus_{i \in I} \pi_i^* L \right) \oplus \left( \bigoplus_{j \in J} \pi_j^* M \right)
\cong
\left( \bigoplus_{i \in \sigma(I)} \pi_i^* L \right) \oplus \left( \bigoplus_{j \in \sigma(J)} \pi_j^* M \right) 
\end{equation}
for all $\sigma\in S_n.$ 

Since $X$ is not an abelian surface, by Lemma \ref{L+L} we have $L \not\cong M$. In particular, at most one of $L, M$ can be trivial. We consider the cases separately.

\textbf{Case 1:} $L, M$ are both nontrivial. So, $\mathcal{O}_X, L, M$ are pairwise nonisomorphic. It easily follows that $\{ \pi_i^* L \}_{i \in [n]}$ and $\{ \pi_j^* M \}_{j \in [n]}$ are pairwise nonisomorphic. Using \eqref{LM} and \cite[Theorem 3]{atiyah1956krull}, we get $\sigma(I) = I$, $\sigma(J) = J$ for all $\sigma \in S_n$. Thus each of $I, J$ is either $[n]$ or $\emptyset$. As $\mathcal{E}_1, \mathcal{E}_2 \neq 0$, we see that $(a)$ holds, possibly after swapping $L$ and $M$.

\textbf{Case 2:} Exactly one of $L, M$, say $L$, is trivial. 

\medskip

Note that $\mathcal{O}_{X^n}, \{ \pi_i^* M \}_{i \in [n]}$ are pairwise nonisomorphic. By \eqref{LM} and \cite[Theorem 3]{atiyah1956krull}, we have $J = \sigma(J)$ for all $\sigma \in S_n.$ So $J = \emptyset$ or $[n]$. If $J = [n]$, then $\operatorname{rank} \mathcal{E}_1 \leq \operatorname{rank} \mathcal{E}_2$ forces $I = \emptyset$. Now swapping $L, M$, we see that $(a)$ holds.

Now suppose $J = \emptyset$. So,
\begin{equation}\label{E1E2}
    \mathcal{E}_1\cong \mathcal{O}_{X^n}^{|I|} \text{ and }
    \mathcal{E}_2\cong \mathcal{O}_{X^n}^{n-|I|}\oplus \left( \bigoplus_{j \in [n]} \pi_j^* M \right)
\end{equation}
We first prove the following claim.

\underline{\textit{Claim:}} The only proper nonzero $S_n$-invariant subbundles of  $ \bigoplus_i \pi_i^* L$  are $\mathcal{F}_1$ and $\mathcal{F}_2$.

\begin{proof}
    Clearly $\mathcal{F}_1$ and $\mathcal{F}_2$ proper nonzero $S_n$-invariant subbundles. we show they are the only ones.
    
Note that for a projective variety $Y$ with an action of a finite group $G$, $G$-linearizations of $\mathcal{O}_Y^n$ are in one-to-one correspondence with homomorphisms $G \to GL(n,\mathbb{C})$.
If $\xi$ is a trivialization of $L$, $\xi$ induces a trivialization of $\bigoplus_i \pi_i^* L$ , and the $S_n$-linearization of $\bigoplus_i \pi_i^* L$ coming from the $S_n$-linearization of $T_{X^n}$ corresponds to the permutation representation $\rho:S_n \to GL(n,\mathbb{C})$.

So the $S_n$-invariant subbundles of $\bigoplus_i \pi_i^* L$ are in one-to-one correspondence with subrepresentations of $\rho$. It is well-known that $\rho$ has exactly two proper nonzero subrepresentations. This shows $\mathcal{F}_1$ and $\mathcal{F}_2$ are the only ones.
\end{proof}
Now we return to the proof of Case $2$.
We analyze two subcases separately.

\textbf{Subcase 1:} $h^0(X, M) = 0$. Hence
\[
\mathrm{Hom}\!\left(\mathcal{E}_1, \bigoplus_{j\in[n]} \pi_j^* M\right)
\cong
\mathrm{Hom}\!\left(\mathcal{O}_{X^n}^{|I|},  \bigoplus_{j\in[n]}\pi_j^* M \right) = 0.
\]
So $\mathcal{E}_1 \subseteq \bigoplus_{i\in[n]} \pi_i^* L$ is $S_n$-invariant subbundle, hence by Claim, $(i)$ of $(b)$ holds.

\medskip

\textbf{Subcase 2:} $h^0(X, M^{-1}) \neq 0$.  As $M$ is nontrivial, we must have $h^0(X, M^{-1}) = 0$. Hence
\[
\mathrm{Hom}\!\left(\bigoplus_{j\in[n]} \pi_j^* M,\bigoplus_{i\in [n]} \pi_i^* L \right) \cong\mathrm{Hom}\!\left(\bigoplus_{j\in [n]} \pi_j^* M, \mathcal{O}_{X^n}^n\right) = 0.
\]
So any subbundle of $T_{X^n}$ isomorphic to $\bigoplus_{j\in [n]} \pi_j^* M$ must be contained in the factor $\bigoplus_{j\in[n]} \pi_j^* M$, hence is equal to that factor. Since $\mathcal{E}_2$ contains a subbundle isomorphic to $\bigoplus_{j\in[n]} \pi_j^* M$, we see that $
\mathcal{E}_2 \supseteq \bigoplus_{j\in[n]} \pi_j^* M.$ So, $\mathcal{E}_2=\mathcal{F}\oplus \left(\bigoplus_{j\in [n]} \pi_j^* M\right)$ for a proper $S_n$-invariant subbundle
$\mathcal{F}$ of $\bigoplus_{i\in[n]} \pi_i^* L$. If $\mathcal{F}=0$, then $(a)$ holds. If $\mathcal{F}\neq 0$, then by Claim $\mathcal{F}=\mathcal{F}_1$ or $\mathcal{F}_2$, so $(ii)$ of $(b)$ holds.

\end{proof}
\section{Non-extendability of some subbundle of $T_{X^n}$ to $X^{[n]}$}
We set up the following notations which will be followed throughout this section. Let $X$ be a smooth projective surface, $n\geq 2$ an integer. Let $f : X^n \to X^{(n)}$ be the Hilbert-Chow morphism, and $q : X^n \to X^{(n)}$ the quotient map, $Z $ the singular locus of $X^{(n)}$, $E=f^{-1}(Z)$, and $\Tilde{Z}=q^{-1}(Z)$, the big diagonals, and $j:X^n\setminus \Tilde{Z}\to X^n$ the inclusion. So,  $f:X^{[n]} \setminus E \cong X^{(n)} \setminus Z$ is an isomorphism and  $q:X^{n} \setminus \Tilde{Z} \to X^{(n)} \setminus Z$ is finite \'etale. Let $\pi_i:X^n\to X$ be the projections.

First we make the following definitions for a smooth projective surface $X$.
\begin{enumerate}
    \item Given a direct summand $\mathcal{E}$ of $T_{X^{[n]}}$, we define an $S_n$-invariant direct summand $\widetilde{\mathcal{E}}$ of $T_{X^n}$ as follows:

Let $\mathcal{F}$ be a subbundle of $T_{X^{[n]}}$ such that $T_{X^{[n]}} = \mathcal{E} \oplus \mathcal{F}.$ Let
\[
\Tilde{\mathcal{E}}' = q^* f_*\big(\mathcal{E}|_{X^{[n]} \setminus E} \big), 
\quad 
\Tilde{\mathcal{F}}' = q^* f_*\big(\mathcal{F}|_{X^{[n]} \setminus E} \big).
\]

Clearly both are subbundles of $T_{X^n \setminus \widetilde{Z}}$ with direct sum $T_{X^n \setminus \widetilde{Z}}$.
Now we define
\[
\widetilde{\mathcal{E}} := j_* \Tilde{\mathcal{E}}', 
\quad 
\widetilde{\mathcal{F}} := j_*\Tilde{\mathcal{F}}' .
\]

Both are naturally subsheaves of $j_* T_{X^n \setminus \widetilde{Z}} = T_{X^n}$, as $\mathrm{codim}_{X^n}(\widetilde{Z}) = 2$.

Also, $\widetilde{\mathcal{E}} \oplus \widetilde{\mathcal{F}} = T_{X^n}$, 
so $\widetilde{\mathcal{E}}$ is a direct summand of $T_{X^n}$, in particular it is a subbundle.

Finally, $\widetilde{\mathcal{E}}$ is $S_n$-invariant, as $\Tilde{\mathcal{E}}|_{X^n \setminus \widetilde{Z}}$ is $S_n$-invariant.
\item Given a subbundle $V$ of $T_{X^n}$, we say $V$ is \emph{extendable} if $V = \widetilde{\mathcal{E}}$ for some direct summand $\mathcal{E}$ of $T_{X^{[n]}}$.
\end{enumerate}
\vspace{1em}
The following result will be used as a crucial step in the proof of Theorem \ref{main}.
\begin{theorem}\label{crucial}
    
Let $X$ be a smooth projective surface, and let $L, M$ be line subbundles of $T_X$ with $T_X = L \oplus M$.
Then the subbundle $\bigoplus_{i=1}^n \pi_i^* M$
of $T_{X^n}$ is not extendable.
\end{theorem} 
\begin{proof}
The idea of the proof is the following: since $f^{-1} \circ q :X^n\setminus{\widetilde{Z}} \to X^{[n]}\setminus E$ is finite \'etale, given a rank $n$ $S_n$-invariant subbundle $V$ of $T_{X^n}$, $V|_{X^n\setminus{\widetilde{Z}}}$ is the pullback of a vector bundle $V'$ on $X^{[n]}\setminus E$ under $f^{-1} \circ q$. Now if $V=\widetilde{\mathcal{E}}$ for a direct summand subbundle $\widetilde{\mathcal{E}}$ of $T_{X^{[n]}}$, then the total space of $\mathcal{E}$ must be the closure of the total space $S$ of $V'$ in  the total space of $T_{X^{[n]}}$. Since we know $V$, we know $S$. We get a contradiction by showing that a fibre of the projection $S\to X^{[n]}$ has dimension $\geq n+1$. This is a local question, so we do the computation in a local chart of $X^{[n]}.$ What we might lose after going to the local chart is the description of the subbundles $L$ and $M$ of $T_X$, but if we use \cite[Theorem C]{beauville2000complex}, we can choose charts in $X$ in such a way that $L$ and $M$ correspond to tangents in horizontal and vertical directions, respectively. So we can get back $L$ and $M$ in the chart.

Now we carry out the idea. For a vector bundle $\mathcal{E}$ on a complex manifold $N$, denote by $\mathrm{Sp}(\mathcal{E})$ the total space of $\mathcal{E}$. Thus, $\mathrm{Sp}(\mathcal{E})$ is a complex manifold with a map to $N$, whose fibres are naturally vector spaces. Given a map $h:N\to N'$, we have an induced map $dh:\mathrm{Sp}(T_N)\to \mathrm{Sp}(T_{N'})$ of complex manifolds.

Let $r_1, r_2 : \mathbb{D}^{2} \to \mathbb{D}$ be the projections. By \cite[Theorem C]{beauville2000complex}, there is open embedding
\[
\phi : \mathbb{D}^{2} \to X
\]
such that
\[
\phi^* L = r_1^* T_{\mathbb{D}}, \phi^* M = r_2^* T_{\mathbb{D}}
\]
as subbundles of $T_{\mathbb{D}^{2}}$. Let $\phi_3,\cdots,\phi_n : \mathbb{D}^{ 2} \to X$ be open embeddings with the same property, such that the images of $\phi, \phi_3,\cdots,\phi_n$ are disjoint.

Note that given a length $2$ subscheme $\mathfrak{t}$ of $\mathbb{C}^2$, there is a unique line $\ell(\mathfrak{t})$ containing $\mathfrak{t}$.
For $y_1,y_2,y_3,y_4 \in \mathbb{C}$, let $\tau(y_1,y_2,y_3,y_4)$
be the unique $\mathfrak{t} \in (\mathbb{C}^2)^{[2]}$ such that:

\begin{enumerate}
\item[(i)] $\ell(\mathfrak{t}) = \{ (z,w) \in \mathbb{C}^2 \mid z = y_1 w + y_2 \}$,
\item[(ii)]If $[\mathfrak{t}] = (z_1,w_1) + (z_2,w_2)$, then $w_1 + w_2 = y_3$ and $(w_1 - w_2)^2 = y_4$.
\end{enumerate}

The map $\mathbb{C}^4 \xrightarrow{\tau} (\mathbb{C}^2)^{[2]}$
is an open embedding. Let $V = \tau^{-1}((\mathbb{D}^{2})^{[2]})$, where $(\mathbb{D}^{2})^{[2]}\subset (\mathbb{C}^2)^{[2]}$ is the analytic open subset consisting of subschemes supported inside $\mathbb{D}^{2}.$ Let
$V_1 = V \times (\mathbb{D}^{2})^{n-2}$,
an open subset of $\mathbb{C}^{2n}$.

We write elements of $V_1$ as $\underline{y}= (y_1,y_2,\dots,y_{2n})$, where $(y_1,y_2,y_3,y_4) \in V_1,$ and $y_i \in \mathbb{D}$ for $i\geq 5$. We will write elements of $\mathrm{Sp}(T_{V_1}) = V_1 \times \mathbb{C}^{2n}$,
as $\underline{y} \times \underline{v'}$, where $\underline{y} \in V_1$, $v' = (v_1',\dots,v_{2n}') \in \mathbb{C}^{2n}, v_i'\in T_{y_i}\mathbb{C} = \mathbb{C}.$

For $ \underline{y}\in V_1$, let $\eta_1(\underline{y}) \in X^{[n]}$ be the disjoint union of
\[
\varphi(\tau(y_1,y_2,y_3,y_4)), \; \varphi_3(y_5,y_6), \dots, \varphi_n(y_{2n-1},y_{2n}).
\]
So, $\eta_1:V_1\to X^{[n]}$ is an open embedding. Let $\Tilde{\mathfrak{s}}=\eta_1(0).$ Let $D = \{ y \in V_1 \mid y_4 = 0 \}$. Note that $\eta_1^{-1}(E) = D$.

Note that, since the map $f^{-1} \circ q :X^n\setminus{\widetilde{Z}} \to X^{[n]}\setminus E$ is finite \'etale, we have a map
\[
d(f^{-1} \circ q) : \mathrm{Sp}(T_{X^n\setminus{\widetilde{Z}}}) \to \mathrm{Sp}(T_{X^{[n]}\setminus E}),
\]
which is also finite \'etale.

Let $S \subset \mathrm{Sp}(T_{X^n})$ be the closure of
\[
d(f^{-1} \circ q)\mathrm{Sp}\left( \bigoplus_i \pi_i^* M \right)
\]
in $\mathrm{Sp}(T_{X^{[n]}})$, and let $g:S \to X^{[n]}$ be
the restriction of the projection $\mathrm{Sp}(T_{X^{[n]}}) \to X^{[n]}$.

Each fibre of $g$ over $y \in X^{[n]}$ is a vector subspace of $T_{X^{[n]}}(y)$, as $\bigoplus_i \pi_i^* M$ is $S_n$-invariant. Hence this holds for all $y \in X^{[n]}$.

The following claim would complete the proof, as if
\[
\bigoplus_i \pi_i^* M = \widetilde{\mathcal{E}}
\]
for some direct summand subbundle $\mathcal{E}$ of $T_{X^{[n]}}$, necessarily of rank $n$, then $\mathrm{Sp}(\mathcal{E})\supset S$ over $X^{[n]}\setminus E$, so $\mathrm{Sp}(\mathcal{E})\supset S$.

\underline{\textit{Claim:}} $\dim g^{-1}(\Tilde{\mathfrak{s}}) \geq n+1$.
\begin{proof} 
Let
\[
V_2 = \{ (z_1,w_1,\dots,z_n,w_n) \in \mathbb{D}^{ 2n} \mid w_1 \neq w_2\}.
\]

We write elements of $V_2$ as $(\underline{z},\underline{w}) = (z_1,w_1,\dots,z_n,w_n)$, elements of $\mathrm{Sp}(T_{V_2}) = V_2 \times \mathbb{C}^{2n}$ as $(\underline{z},\underline{w})\times \underline{v}$, where $\underline{v} = (u_1,v_1,u_2,v_2,\dots,u_n,v_n), u_i \in T_{z_i}\mathbb{D} = \mathbb{C}, \; v_i \in T_{w_i}\mathbb{D} = \mathbb{C}$.

Define $\beta_i : V_2 \to \mathbb{D}$, by $\beta_i(\underline{z},\underline{w}) = w_i$. Define $\eta_2 : V_2 \to X^n\setminus \Tilde{Z}$ by
\[
\eta_2(\underline{z},\underline{w}) = \big( \varphi(z_1,w_1), \varphi(z_2,w_2),\phi_3(z_3,w_3), \dots, \phi_n(z_n,w_n) \big).
\]
Clearly, $\eta_2$ is an open embedding, and
\[
\eta_2^* \pi_i^* M = \beta_i^* T_\mathbb{D}
\]
as subbundles of $T_{V_2}$, for $1\leq i\leq n$.

So,
\[
\mathrm{Sp}\big(\bigoplus_i \beta_i^* T_\mathbb{D}\big) = (d\eta_2)^{-1} \left( \mathrm{Sp}\big(\bigoplus_i \pi_i^* M \big|_{X^n\setminus \Tilde{Z}}\big) \right).
\]

Define $\pi : V_2 \longrightarrow V_1\setminus D$
by $\pi(\underline{z},\underline{w}) = \underline{y}$, where $\underline{y} = (y_1, \dots, y_{2n})$ is given by
\begin{align*}
y_1 &= \frac{z_2 - z_1}{w_2 - w_1}, \\
y_2 &= \frac{z_1 w_2 - z_2 w_1}{w_2 - w_1}, \\
y_3 &= w_1 + w_2, \\
y_4 &= w_1 w_2,
\end{align*}
and
\[
y_{2i-1} = z_i, \quad y_{2i} = w_i \quad \text{for } 3 \leq i \leq n.
\]
One can check that the following diagram commutes:
  \begin{center}
\begin{tikzcd}
V_2 \arrow[r, "\eta_2"] \arrow[d,"\pi"]
& X^n\setminus \Tilde{Z} \arrow[d, "f^{-1}\circ q" ] \\
V_1\setminus D  \arrow[r,"\eta_1" ]
& |[, rotate=0]| X^{[n]}\setminus E
\end{tikzcd}.
\end{center}

So, we have a commutative diagram:
\[
\begin{tikzcd}
Sp\big(\bigoplus_i\beta_i^*T_{\mathbb{D}}\big) \arrow[r] \arrow[d,hook] & Sp\big(\bigoplus_i\pi_i^*M\big) \arrow[d,hook] \\
Sp(T_{V_2}) \arrow[r, "d\eta_2"] \arrow[d, "d\pi"] & Sp(T_{X^n\setminus \Tilde{Z}}) \arrow[d, "d(f^{-1}\circ q)"] \\
Sp(T_{V_1\setminus D}) \arrow[r, "d\eta_1"] \arrow[d] & Sp(T_{X^{[n]}\setminus E}) \arrow[d] \\
Sp(T_{V_1}) \arrow[r, "d\eta_1"] & Sp(T_{X^{[n]}})
\end{tikzcd},
\]
whose all rows are open embeddings, the columns of the bottom square are inclusion maps and the top and bottom squares are Cartesian.

Note that $S$ is the closure of the image of $\mathrm{Sp}\big(\bigoplus_i\pi_i^*M\big)$ under the composition of the maps in the right column. Let $S_1$ be the closure of the image of $\mathrm{Sp}\big(\bigoplus_i\beta_i^*T_{\mathbb{D}}\big)$ under the composition of the maps in the left columnn, and let $g_1$ be the restriction to $S_1$ of the projection $Sp(T_{V_1})=V_1\times\mathbb{C}^{2n}\to V_1.$ As $\eta_1(0)=\Tilde{\mathfrak{s}},$ it suffices to show $\dim g_1^{-1}(0) \geq n+1$. This will follow if we can show the following statement:

Let $\underline{v'} = (v_1', v_2', \dots, v_{2n}') \in \mathbb{C}^{2n}$
be such that
\[v_2' = 0, \quad v_{2i+1}' = 0 \ \text{for } 2 \leq i \leq n-1, \quad v_1', v_{4}' \neq 0.
\]
Then $\underline{0} \times \underline{v'} \in S_1$.

We now proceed to prove the above statement. First note that if $\underline{v} = (u_1, v_1, \dots, u_n, v_n) \in \mathbb{C}^{2n}$
is such that $u_i = 0$ for all $i$, then $d\pi((\underline{z},\underline{w})\times \underline{v}) = \underline{y} \times \underline{v'},$
where $y = \pi(\underline{z},\underline{w})$ and $\underline{v'} = (v_1', \dots, v_{2n}')$ is given by
\begin{align*}
v_1' &= -\frac{z_2 - z_1}{(w_2 - w_1)^2}(v_2 - v_1), \\
v_2' &= \frac{z_2 - z_1}{(w_2 - w_1)^2}(w_1 v_2 - w_2 v_1), \\
v_3' &= v_1 + v_2, \\
v_4' &= 2(w_2 - w_1)(v_2 - v_1), 
\end{align*}
and
$$
v_{2i-1}' = 0, v_{2i}' = v_i
$$ for $3\leq i\leq n.$

Now let $y_1 \neq 0$ with $|y_1|$ sufficiently small. Define $\underline{y} = (y_1, \dots, y_{2n}) \in V_1 \setminus D$
by
\[
y_2 = 0, \quad y_3 = -y_1 \frac{v_3'}{v_1'}, \quad y_4 = -\frac{1}{2} y_1 \frac{v_4'}{v_1'}, \quad y_i = 0 \ \text{for } 5 \le i \le 2n.
\]

It suffices to show $\underline{y} \times \underline{v'} \in d\pi\!\left( \mathrm{Sp}\!\left( \bigoplus_i \beta_i^* T_{\mathbb{D}} \right) \right)$.

 There exist $w_1, w_2 \in \mathbb{C}$ such that
\[
w_1 + w_2 = y_3, \quad (w_1 - w_2)^2 = y_4.
\]
Define $w_i = 0$ for $3 \le i \le n$. Thus
\[
\underline{w} = (w_1, \dots, w_n) \in \mathbb{D}^{ n},
\]
for $|y_1|$ sufficiently small.

Let $\underline{z} = y_1 \underline{w}$. Define $v_1, v_2 \in \mathbb{C}$ by
\[
v_1 + v_2 = v_3', \quad v_2 - v_1 = \frac{v_4'}{2(w_2 - w_1)}.
\]
Let $v_i = v_{2i}'$ for $3 \le i \le n$, and $u_i = 0$ for $1\leq i\leq n$. Thus we get $\underline{v} = (u_1,v_1,...,u_n,v_n).$ Since $u_i = 0$ for all $i$, we have
\[
(\underline{z},\underline{w}) \times \underline{v} \in \mathrm{Sp}\big(\bigoplus_i\beta_i^*T_{\mathbb{D}}\big) \subseteq \mathrm{Sp}(T_{V_2}).
\]
Now from  the formula of $d\pi$, one can check that $d\pi((\underline{z},\underline{w})\times \underline{v}) = \underline{y} \times \underline{v'}.$ Thus  $\underline{y} \times \underline{v'} \in d\pi\!\left( \mathrm{Sp}\!\left( \bigoplus_i \beta_i^* T_{\mathbb{D}} \right) \right)$, and we are done.
\end{proof}
\end{proof}

\section{Decomposition of tangent bundle of punctual Hilbert scheme }
In this section, we shall prove the following theorem, which is a more precise version of Theorem \ref{A}.
\begin{theorem}\label{main}
    Let $X$ be a smooth projective surface, $n \geq 2$ an integer. Then the following holds:
\begin{enumerate}
\item If $X$ is an abelian surface, then
\[
T_{X^{[n]}} \cong \mathcal{O}_{X^{[n]}}^{\oplus 2} \oplus \mathcal{E}
\]
for an indecomposable vector bundle $\mathcal{E}$. 

\item If $X$ is of class $\tilde{\mathcal{C}}$., then
\[
T_{X^{[n]}} \cong \mathcal{O}_{X^{[n]}} \oplus \mathcal{E}
\]
for an indecomposable vector bundle $\mathcal{E}$.

\item In all other cases, $T_{X^{[n]}}$ is indecomposable.

Further, in $(1)$ and $(2)$, there is a smooth morphism $\pi:X^{[n]}\to F'$ with $F'$ an abelian variety such that $\mathcal{E} = \ker\!\left( T_{X^{[n]}} \xrightarrow{d\pi} \pi^* T_{F'} \right).$
\end{enumerate}
\end{theorem}
We first prove the following Lemmas.
\begin{lemma}\label{basic}
    Let $X,Y$ be complex manifolds, and $G$ a group acting on $X \times Y$
by diagonal covering space action. Let
\[
X \times Y \xrightarrow{q} Z
\]
be the quotient under the action of $G$. Let $\pi_1:X\times Y\to X, \pi_2:X\times Y\to Y$ be the projections. Then there are unique subbundles
$E,F$ of $T_Z$ such that $E \oplus F = T_Z,$
and
\[
q^*E = \pi_1^* T_X \oplus 0, \quad q^*F = 0 \oplus \pi_2^* T_Y
\]
as subbundles of $T_{X \times Y}$.

Further, if the action of $G$ on $Y$ is a covering space action so that $Y/G$ is a complex manifold, then $F\cong f^*T_{Y/G}$, where $f:Z\to Y/G$ is the induced map.

\end{lemma}
\begin{proof}
    One can locally define $E$ and $F$ satisfying the above properties. They glue globally to subbundles because the $G$ acts diagonally on $X\times Y.$
\end{proof}
\begin{lemma}\label{uniform}
Let $X$ be an abelian surface or of class $\tilde{\mathcal{C}}$.
Then there is a complex manifold $E$, an commutative complex Lie group $F$,
and a discrete subgroup $G \subseteq F$ with an action of $G$ on $E$
such that if $G$ acts on $E \times F$ via the diagonal action,
where $G$ acts on $F$ by translation, then
\[
X \cong (E \times F)/G.
\]
\end{lemma}
\begin{proof}
We consider $3$ cases separately.

\textbf{Case 1:} $X$ abelian surface. Then we can take $
G = 0, F = X$ and $E$ a point.

\textbf{Case 2:} $X$ is bielliptic or of class $\tilde{\mathcal{C}}$. In this case Lemma is immediate from the definition, with $F$ an elliptic curve.

\textbf{Case 3:} $X$ is a $\mathbb{P}^1$-bundle over an elliptic curve $F$, and $T_X$ is split.
Note that 
 the universal cover of $X$ is isomorphic to $\mathbb{C} \times \mathbb{P}^1$. By \cite[Theorem C]{beauville2000complex}, there is a covering map
\[
\mathbb{C} \times \mathbb{P}^1 \xlongrightarrow{\pi} X
\]
such that $\pi$ is a quotient map under a diagonal covering space action of $G := \pi_1(X)\cong \mathbb{Z}^2$ on $\mathbb{C} \times \mathbb{P}^1$.

    If for some $1 \neq g \in \pi_1(X)$, the action of $g$ on $\mathbb{C}$ has a fixed point $z_0$, then $g$ gives an automorphism of $\{z_0\} \times \mathbb{P}^1 \cong \mathbb{P}^1$,
and hence has a fixed point there. This contradicts that $G$ acts freely on $\mathbb{C} \times \mathbb{P}^1$. So, $G$ acts freely on $\mathbb{C}$, hence $G$ acts by translations on $\mathbb{C}$. So, the orbit of $0$ in $\mathbb{C}$ embeds $G$ as a lattice in $\mathbb{C}$.
Thus, we can take $E = \mathbb{P}^1$ and $F = \mathbb{C}$.
\end{proof}
Now we are ready to prove Theorem \ref{main}.

\textit{Proof of Theorem \ref{main}:}
\begin{proof}

\underline{\textit{Proof of (1) and (2):}}
Identify $X = (E \times F)/G$ as in Lemma \ref{uniform}. Let $F'=F/G$, and $F \xrightarrow{q} F'$ the quotient map, and $p:X \to F'$ be the induced map. Note that $F'$ is either an abelian surface or an elliptic curve. For $n\in \mathbb{Z}$, let $F \xrightarrow{nq} F'$ be the composition of $q$ with the multiplication by $n$ map in $F'$. For $(e,f)\in E\times F$, we denote $\overline{(e,f)}$ the image of $(e,f)$ in $X$, also we abbreviate $q(f)$ as $\overline{f}$.

Let $\oplus, \ominus$ denote addition and subtraction on $F$ or $F'$.
We have a well-defined action of $F$ on $X$ given by
\[
f' \cdot \overline{(e,f)} = \overline{(e, f' + f)}.
\]

Let $X^{[n]} \xrightarrow{\pi} F'$
be defined as follows:

Given a length $n$ subscheme $Z$ of $X$, regarded as a point of $X^{[n]}$,
let
\[
p_*[Z] = \sum_i \bar{f_i}
\]
be the pushforward of the corresponding $0$-cycle.
Then define
\[
\pi(Z) = \oplus_{i=1}^n \bar{f_i}.
\]
Let $K=\pi^{-1}(0).$

Let
\begin{center}
\begin{tikzcd}
Y \arrow[r, "g"] \arrow[d,""]
& X^{[n]} \arrow[d, "\pi" ] \\
F  \arrow[r,"nq" ]
& |[, rotate=0]| F'
\end{tikzcd}.
\end{center}

be fiber square.
that is,
\[
Y = \{ (Z,f) \in X^{[n]} \times F \mid \pi(Z) = n \overline{f} \}.
\]

Note that $K\times F\xrightarrow{\phi}Y$ given by $$\phi(W,f)=(f\cdot W,f)$$ is an isomorphism over $F$, with inverse
\[
\phi^{-1}(Z,f) = ((-f)\cdot Z, f).
\]

Let $H = \ker(nq) \subset F$, a finite subgroup of $F$. As $F \xrightarrow{nq} F'$
is the quotient of $F$ under the action of $H$ by translations, so $Y \xrightarrow{g} X^{[n]}$
is the quotient of $Y$ under the action of $H$. Via the isomorphism $\phi$, this action is the following: given $f' \in H$ and $(W,f) \in K \times F$, we have
\[
f' \cdot (W,f) = \bigl((-f')\cdot W,\, f \oplus f'\bigr).
\]
Since this is a diagonal action, Lemma \ref{basic} gives subbundles $\mathcal{E}_1, \mathcal{E}_2$ of $T_{X^{[n]}}$ with $T_{X^{[n]}} = \mathcal{E}_1 \oplus \mathcal{E}_2$. 
By Lemma \ref{basic}, we have 
$$ \mathcal{E}_1\cong \pi^*T_{F'}=
\begin{cases}
\mathcal{O}_{X^{[n]}}^{\oplus 2} , & \quad \textnormal{if $X$ is an abelian surface}, \\
\mathcal{O}_{X^{[n]}}, & \quad \textnormal{if $X$ is of class $\tilde{\mathcal{C}}$},

\end{cases}.$$

and
\[
\mathcal{E}_2 = \ker\!\left( T_{X^{[n]}} \xrightarrow{d\pi} \pi^* T_{F'} \right).
\]

\medskip

Let $\mathcal{E} = \mathcal{E}_2.$ It remains to show $\mathcal{E}$ is indecomposable.

If $X=A$ is an abelian surface, then $K=Kum_n(A)$ is hyperkähler manifold, so by \cite[Section 2.5]{anella2022effectivity}, $T_K$ is stable with respect to any Kähler class, hence indecomposable. As $\mathcal{E}|_K\cong T_K$, we see that $\mathcal{E}$ is indecomposable.

Now assume $X$ is of class $\Tilde{\mathcal{C}}$.
Let $\mathcal{E}_1 = \mathcal{L}$, a line subbundle of $T_{X^{[n]}}$. Suppose $\mathcal{E}$ is not indecomposable.

Write $\mathcal{E} = \mathcal{G}_1 \oplus \mathcal{G}_2$,
where $\mathcal{G}_1, \mathcal{G}_2$ are subbundles of $\mathcal{E}$ of ranks $a$ and $b$ respectively, with $1 \le a \le b$. As $a + b = 2n - 1,$ we have $a < b$.

We have
\[
T_{X^n} = \widetilde{\mathcal{L}} \oplus \widetilde{\mathcal{F}}_1 \oplus \widetilde{\mathcal{F}}_2.
\]

Let $L, M$ be subbundles of $T_X$ with $L \oplus M = T_X$.

By Lemma \ref{crucial}, none of the subbundles
\[
 \widetilde{\mathcal{L}},\widetilde{\mathcal{G}}_1, \widetilde{\mathcal{G}}_2, \widetilde{\mathcal{L}} \oplus \widetilde{\mathcal{G}}_1,\widetilde{\mathcal{L}} \oplus \widetilde{\mathcal{G}}_2,
\widetilde{\mathcal{G}}_1 \oplus \widetilde{\mathcal{G}}_2
\]
can be $\bigoplus_i \pi_i^* L$ or $\bigoplus_i \pi_i^* M$.

So, possibly after swapping $L, M$, $(b)$ of Lemma \ref{Sn invariant} holds for each of the pairs
\[
(\widetilde{\mathcal{L}}, \widetilde{\mathcal{G}}_1 \oplus \widetilde{\mathcal{G}}_2), 
\quad 
(\widetilde{\mathcal{G}}_1, L \oplus \widetilde{\mathcal{G}}_2), 
\quad 
(\widetilde{\mathcal{L}} \oplus \widetilde{\mathcal{G}}_1, \widetilde{\mathcal{G}}_2).
\]

If $h^0(X, M) = 0$, then $\widetilde{\mathcal{L}}, \widetilde{\mathcal{G}}_1$, $\widetilde{\mathcal{L}} \oplus \widetilde{\mathcal{G}}_1$ are three distinct subbundles, but each equals either $\mathcal{F}_1$ or $\mathcal{F}_2$, a contradiction.

If $h^0(X, M) \neq 0$, then $\widetilde{\mathcal{G}}_1 \oplus \widetilde{\mathcal{G}}_2$, $\mathcal{L} \oplus \widetilde{\mathcal{G}}_2$, $\widetilde{\mathcal{G}}_2$ are three distinct subbundles, each equal to either
\[
\mathcal{F}_1 \oplus \left( \bigoplus_j \pi_j^* M \right)
\quad \text{or} \quad
\mathcal{F}_2 \oplus \left( \bigoplus_j \pi_j^* M \right),
\]
again a contradiction.

\underline{\textit{Proof of $(3)$:}} We consider two cases separately.

\textbf{Case 1:} $T_X$ is split. Say $T_X = L \oplus M$, with $L, M$ line subbundles.
Suppose $T_{X^{[n]}} = \mathcal{E}_1 \oplus \mathcal{E}_2,$
with $\mathcal{E}_1, \mathcal{E}_2$ nonzero subbundles. So $T_{X^{[n]}} = \Tilde{\mathcal{E}_1} \oplus \Tilde{\mathcal{E}_2}.$ As $X$ is not an abelian surface, and $L, M$ are nontrivial by Lemma \ref{O+K}, by Lemma \ref{Sn invariant}, we see that, after swapping $L, M$ if necessary, $\widetilde{\mathcal{E}}_2 = \bigoplus_j \pi_j^* M.$
This contradicts Lemma \ref{crucial}.

\textbf{Case 2:} $T_X$ is indecomposable.

Suppose $T_{X^{[n]}} = \mathcal{E}_1 \oplus \mathcal{E}_2$, with $\mathcal{E}_1, \mathcal{E}_2$ nonzero subbundles. Then $T_{X^n} = \widetilde{\mathcal{E}}_1 \oplus \widetilde{\mathcal{E}}_2$,
with $\widetilde{\mathcal{E}}_1, \widetilde{\mathcal{E}}_2$ nonzero $S_n$-invariant subbundles of $T_{X^n}$. This contradicts Theorem \ref{Sn invariant1}.
\end{proof}
\section{Distinguishing products of varieties}
We prove Theorems \ref{conjecture} and \ref{C} in this section.

\textit{Proof of Theorem \ref{conjecture}}:
\begin{proof}
    
Let
\[
Y = \prod_i X^{[a_i]}, \quad \pi_i : Y \to X^{[a_i]}
\]
be the projections.

Let $$e=
\begin{cases}
2 , & \quad \textnormal{if $X$ is an abelian surface}, \\
1, & \quad \textnormal{if $X$ is of class $\tilde{\mathcal{C}}$}, \\
0 , & \quad \textnormal{otherwise}.
\end{cases}$$

For $k \geq 2$, let $t_k$ be the number of summands of rank $2k - e$ in a decomposition of $T_Y$ as a direct sum of indecomposable vector bundles. By \cite[Theorem 3]{atiyah1956krull}, $t_k$ is well-defined, that is, independent of the decomposition.

On the other hand,
\[
T_Y \cong \bigoplus_i \pi_i^* T_{X^{[a_i]}},
\]
and Theorem \ref{main} shows $t_k$ is the number of times $k$ occurs in the multiset $\{a_1,a_2,\cdots, a_r\}$. Here we are using the observation: if $\mathcal{E}$ is an indecomposable vector bundle on $X^{[a_i]}$, then $\pi_i^*(\mathcal{E})$ is indecomposable, as its restriction to a section of $\pi_i$ is $\mathcal{E}$, hence indecomposable. As $Y\cong \prod_j X^{[b_j]}$, we see that $t_k$ is also same as the number of times $k$ occurs in the multiset $\{b_1,b_2,\cdots, b_r\}$. So, for each $k\geq 2$, $k$ occurs same number of times in each multiset. Now $$\sum_ia_i=\frac{1}{2}\dim Y=\sum_j b_j$$ shows this is true for $k=1$ also. So both multisets are the same.
\end{proof}
\begin{remark}\label{general}
    For any smooth projective surface $X$ one can define the invariant $e(X)$ as in the proof of Theorem \ref{conjecture}. Note that by Lemmas \ref{L+L} and \ref{O+K}, $e(X)$ is the largest rank of a trivial direct summand of $T_X$. The proof above in fact shows the following more general statement: Let $X_i, Y_j$ be smooth projective surfaces for $1\leq i\leq r, 1\leq j\leq s$, with $e(X_i), e(Y_j)$'s all equal, and let $a_1, \dots, a_r$, $b_1, \dots, b_s$ be positive integers.
Suppose
\[
\prod_i X_i^{[a_i]} \cong \prod_j Y_j^{[b_j]}.
\]
Then $r=s$ and
$ \{a_1, \dots, a_r\} = \{b_1, \dots, b_s\}
$
as multisets.
\end{remark}

The conjecture in \cite{dey2026classification} was in fact the statement of Theorem \ref{conjecture} without the assumption that $X$ is connected. We show that this also follows by some manipulation with the connected components.
\begin{corollary}\label{disconnected}
    Let $X$ be a smooth projective $\mathbb{C}$-scheme of pure dimension $2$, and $a_1, \dots, a_r$, $b_1, \dots, b_s$ positive integers.
Suppose
\[
\prod_i X^{[a_i]} \cong \prod_j X^{[b_j]}.
\]
Then $r=s$ and
$ \{a_1, \dots, a_r\} = \{b_1, \dots, b_s\}
$
as multisets.
\end{corollary}
\begin{proof}
    Let \( X_1, X_2, \dots, X_k \) be the connected components of \( X \). Assume \( r \le s \) without loss of generality. For \( n \in \mathbb{N} \), note that the connected components of \( X^{[n]} \) are $\prod_{i=1}^k X_i^{[c_i]} \quad \text{for } \sum_{i=1}^k c_i = n.$ Here \( X_i^{[0]} \) is a point by convention.

Since $\prod_{i=1}^r X_1^{[a_i]}$
is a connected component of $\prod_{i=1}^r X^{[a_i]}$,
there is a connected component of $\prod_{i=1}^s X^{[b_i]}$
isomorphic to it. Thus, there are nonnegative integers \( b_{ij} \) for \( 1 \le i \le s \), \( 1 \le j \le k \), with $\sum_{j=1}^k b_{ij} = b_i \quad \text{for all } i$, 
and
\begin{equation}\label{long}
    \prod_{i=1}^r X_1^{[a_i]} \cong \prod_{\substack{1 \le i \le s \\ 1 \le j \le k}} X_j^{[b_{ij}]}.
\end{equation}

Let \( A \) be the multiset \( \{a_1, \dots, a_r\} \), and \( B \) be the multiset
\[
\{ b_{ij} \mid 1 \le i \le s,\, 1 \le j \le k,\, b_{ij} \neq 0 \}.
\]
Clearly, $|B| \ge s \ge r = |A|$,
and if \( |B| = s \), then $B = \{ b_1, \dots, b_s \}$.
Thus, it suffices to show \( A = B \) as multisets. We consider the following cases separately.

\textbf{Case 1:} All \( X_i \)'s are abelian surfaces. In this case we are done by Remark \ref{general}.

\textbf{Case 2:} There is some \( i \) with \( X_i \) neither an abelian surface nor of class $\Tilde{\mathcal{C}}$.

Without loss of generality assume \( i = 1 \). The number of indecomposable components of the tangent bundle of the left-hand side of \eqref{long} which are trivial line bundles is \( 0 \). By \cite[Theorem 3]{atiyah1956krull}, the same has to be true for the right-hand side. This forces $e(X_j) = 0$ 
unless \( b_{ij} = 0 \) for all \( i \). Now we are done by Remark \ref{long}.

\noindent\textbf{Case 3:} Each \( X_i \) is either an abelian surface or of class $\Tilde{\mathcal{C}}$, and not all are abelian surfaces.

Without loss of generality assume \( X_1 \) is of class $\Tilde{\mathcal{C}}$. The number of indecomposable components of the tangent bundle of the left-hand side of \eqref{long} which are trivial line bundles is \( r \). For the right-hand side, this number is \( \ge |B| \), and equality holds if and only if:
\begin{equation}\label{equality}
    e(X_j) = 1, \text{ unless } b_{ij} = 0  \text{ for every } i .
\end{equation}

Since \( |B| \ge r \), \cite[Theorem 3]{atiyah1956krull} shows \( |B| = r \), and \eqref{equality} holds. Now we are done by Remark \ref{general}.

\end{proof}
Now we prove Theorem \ref{C}.

\textit{Proof of Theorem \ref{C}:} 

\begin{proof}
    We induct on $\sum_i a_i=\sum_j b_j$, which is $\dim(X)^{-1}$ times the dimension of the product space. The base case is trivial. Let $Z=\prod_i X^{(a_i)}$. For $m \in \mathbb{N}$, let $k_m$ be the number of $i$ with $a_i = m$, and let $\ell_m$ be the number of $j$ with $b_j = m$. So
\[
Z \cong \prod_m (X^{(m)})^{k_m} \cong \prod_m (X^{(m)})^{\ell_m}.
\]
What we want to show is that $\underline{k} = \underline{\ell}$.

For $m \ge 2$, let $W_m = \operatorname{Sing} (X^{(m)})$. By \cite[Theorem 2]{bansal2025symmetric}, $W_m$ is irreducible with normalization $X \times X^{(m-2)}$.

The intersection of the irreducible components of $\operatorname{Sing} Z$ is isomorphic to $X^{k_1}\times\prod_{m \ge 2} W_m^{k_m}$, and also to $X^{l_1}\times\prod_{m \ge 2} W_m^{l_m}$.

Taking normalizations, we get
\[
X^{\sum_m k_m + k_3} \times \prod_{m \ge 2} (X^{(m)})^{k_{m+2}}
\cong
X^{\sum_m l_m + l_3} \times \prod_{m \ge 2} (X^{(m)})^{l_{m+2}}.
\]

By induction hypothesis, 
\begin{equation}\label{a0}
k_m = \ell_m \text{ for all } m \ge 4,
\end{equation}
 and we have 
\begin{equation}\label{a1}
    k_1 + k_2 + 2k_3 = \ell_1 + \ell_2 + 2\ell_3.
\end{equation}

Also,
\[
\dim Z = \left(\sum_m m k_m\right)\dim X = \left(\sum_m m \ell_m\right)\dim X,
\]
hence
\begin{equation}\label{a2}
    k_1 + 2k_2 + 3k_3 = \ell_1 + 2\ell_2 + 3\ell_3.
\end{equation}

For $n \ge 2$ with $k_n \ne 0$, let
\[
G_n = \prod_{m \ne n} (X^{(m)})^{k_m} \times X \times X^{(n-2)}\times (X^{(n)})^{k_n-1}.
\]

For $n \ge 2$ with $\ell_n \ne 0$, 
\[
H_n = \prod_{m \ne n} (X^{(m)})^{l_m} \times X \times X^{(n-2)}\times (X^{(n)})^{l_n-1}.
\] Here $X^{(0)}$ is a point, by convention. By \cite[Theorem 2]{bansal2025symmetric}, $G_n$'s are the normalizations of irreducible components of the singular locus of $\prod_m (X^{(m)})^{k_m}$, similarly for $H_n$'s.

Suppose $\underline{k} \ne \underline{\ell}$. We want to get a contradiction. 
We first prove the following Claim.

\underline{\textit{Claim:}} Suppose $n \in \mathbb{N}$ is such that $k_{n'} = \ell_{n'}$ for all $n' > n$. Then
\[
(k_n,\ell_n) \in \{(0,1),(1,0),(0,0)\}.
\]

\begin{proof}
    As $\underline{k} \ne \underline{\ell}$, \eqref{a1} and \eqref{a2} force $n \ge 3$. Suppose $(k_n,\ell_n)$ is not one of the above. One of $k_n,\ell_n$ is nonzero; assume without loss of generality $k_n \ne 0$.

As $G_n$ is normalization of an irreducible component of $Z$, there is $n'\in\mathbb{N}$ with $\ell_{n'} \ne 0$ with $G_n \cong H_{n'}.$

If $n' > n$, then equating the exponents of $X^{(n)}$ in $G_n$ and $H_{n'}$, by induction hypothesis we get $k_{n'} = \ell_{n'} - 1,$ a contradiction. If $n' = n$, then equating the exponent of $X$ in $G_n$ and $H_n$ by induction hypothesis, we get $k_1 = \ell_1.$
But then by \eqref{a0},\eqref{a1} and \eqref{a2} we get $\underline{k} = \underline{\ell}$, a contradiction.

So, $n' < n$. Equating the exponent of $X^{(n)}$ on both sides $G_{n}, H_{n'}$ by induction hypothesis, we get
\begin{equation}\label{a4}
    k_{n} - 1 = \ell_{n}.
\end{equation}

If $\ell_{n} \ne 0$, then the same argument, with $k$ and $\ell$ swapped, shows $\ell_{n} - 1 = k_{n}$,
a contradiction to \eqref{a4}. So, $\ell_{n} = 0$ and $k_{n} = 1$.
\end{proof} 
By \eqref{a0}, we can apply the claim to each $n \ge 3$. So we obtain $k_n = \ell_n = 0$ for all $n \ge 4$,
and up to swapping $\underline{k}$ and $\underline{\ell}$, we have $k_3 = 1$, $\ell_3 = 0$.

So, $X^{k_1} \times (X^{(2)})^{k_2}\times W_3$  is an irreducible component of $\operatorname{Sing} Z$, and is nonnormal, as $W_3$ is nonnormal: by \cite[Theorem 2]{bansal2025symmetric}, $W_3$ is singular but the normalization of $W_3$ is $X\times X$ which is smooth.

But every irreducible component of $X^{l_1} \times (X^{(2)})^{l_2}$ is normal, as $W_2$ is normal (in fact $W_2\cong X$). A contradiction.
\end{proof}
\begin{remark}
A similar proof would also show the following more general statement: Let $X_i, Y_j$ be smooth varieties for $1\leq i\leq r, 1\leq j\leq s$, with $\dim X_i, \dim Y_j$'s all equal and $\geq 2$, and let $a_1, \dots, a_r$, $b_1, \dots, b_s$ be positive integers.
Suppose
\[
\prod_i X_i^{(a_i)} \cong \prod_j Y_j^{(b_j)}.
\]
Then $r=s$ and
$ \{a_1, \dots, a_r\} = \{b_1, \dots, b_s\}
$
as multisets.
\end{remark}
\begin{remark}
    Theorem \ref{C} does not hold in general for $\dim X=1$, as for example if $X=\mathbb{A}^1$, then $\prod_i X^{(a_i)}\cong\mathbb{A}^{\sum_ia_i}$ always. However, it holds when $X$ is a smooth projective curve, as proved in \cite{mukherjee2025diagonal}.
\end{remark}

 \section{Acknowledgement}
 I thank Prof. János Kollár for giving valuable ideas and insightful discussions.
\printbibliography

@article{beauville2000complex,
  title={Complex manifolds with split tangent bundle},
  author={Beauville, Arnaud},
  journal={Complex analysis and algebraic geometry},
  pages={61--70},
  year={2000}
}

@article{atiyah1956krull,
  title={On the Krull-Schmidt theorem with application to sheaves},
  author={Atiyah, Michael F},
  journal={Bulletin de la Soci{\'e}t{\'e} math{\'e}matique de France},
  volume={84},
  pages={307--317},
  year={1956}
}

@article{popa2014kodaira,
  title={Kodaira dimension and zeros of holomorphic one-forms},
  author={Popa, Mihnea and Schnell, Christian},
  journal={Annals of Mathematics},
  pages={1109--1120},
  year={2014},
  publisher={JSTOR}
}

@inproceedings{fong2024connected,
  title={Connected algebraic groups acting on algebraic surfaces},
  author={Fong, Pascal},
  booktitle={Annales de l'Institut Fourier},
  volume={74},
  number={2},
  pages={545--587},
  year={2024}
}

@article{dey2026classification,
  title={On the classification of products of Hilbert schemes of points over a surface},
  author={Dey, Arijit and Mukherjee, Arijit and Pahari, Anubhab},
  journal={arXiv preprint arXiv:2604.01374},
  year={2026}
}

@article{anella2022effectivity,
  title={Effectivity of semi-positive line bundles},
  author={Anella, F and Huybrechts, D},
  journal={Milan Journal of Mathematics},
  volume={90},
  number={2},
  pages={389--401},
  year={2022},
  publisher={Springer}
}

@article{fogarty1968algebraic,
  title={Algebraic families on an algebraic surface},
  author={Fogarty, John},
  journal={American Journal of Mathematics},
  volume={90},
  number={2},
  pages={511--521},
  year={1968},
  publisher={JSTOR}
}

@article{belmans2020automorphisms,
  title={Automorphisms of Hilbert schemes of points on surfaces},
  author={Belmans, Pieter and Oberdieck, Georg and Rennemo, J{\o}rgen},
  journal={Transactions of the American Mathematical Society},
  volume={373},
  number={9},
  pages={6139--6156},
  year={2020}
}

@article{Ha,
  title={Automorphisms of the Hilbert schemes of $n$ points of a rational surface and the anticanonical Iitaka dimension},
  author={Hayashi, Taro},
  journal={Geometriae Dedicata},
  volume={207},
  pages={395--407},
  year={2020},
  publisher={Springer}
}

@article{Og,
  title={On automorphisms of the punctual Hilbert schemes of K3 surfaces},
  author={Oguiso, Keiji},
  journal={European Journal of Mathematics},
  volume={2},
  number={1},
  pages={246--261},
  year={2016},
  publisher={Springer}
}

@article{Wa,
  title={On automorphisms of Hilbert squares of smooth hypersurfaces},
  author={Wang, Long},
  journal={Communications in Algebra},
  volume={51},
  number={2},
  pages={586--595},
  year={2023},
  publisher={Taylor \& Francis}
}

@article{bansal2025automorphisms,
  title={Automorphisms of punctual Hilbert schemes and symmetric powers of varieties},
  author={Bansal, Ashima and Sarkar, Supravat and Vats, Shivam},
  journal={arXiv preprint arXiv:2508.18059},
  year={2025}
}

@article {wandel2013stability,
    AUTHOR = {Wandel, Malte},
     TITLE = {Stability of tautological bundles on the {H}ilbert scheme of
              two points on a surface},
   JOURNAL = {Nagoya Math. J.},
  FJOURNAL = {Nagoya Mathematical Journal},
    VOLUME = {214},
      YEAR = {2014},
     PAGES = {79--94},
      ISSN = {0027-7630,2152-6842},
   MRCLASS = {14F05 (14C05 14D20 14J28)},
  MRNUMBER = {3211819},
MRREVIEWER = {Maria\ Luisa\ Spreafico},
    
}

@article{scala2015higher,
  title={Higher symmetric powers of tautological bundles on Hilbert schemes of points on a surface},
  author={Scala, Luca},
  journal={arXiv preprint arXiv:1502.07595},
  year={2015}
}

@article{oprea2022big,
  title={Big and nef tautological vector bundles over the Hilbert scheme of points},
  author={Oprea, Dragos and others},
  journal={SIGMA. Symmetry, Integrability and Geometry: Methods and Applications},
  volume={18},
  pages={061},
  year={2022},
  publisher={SIGMA. Symmetry, Integrability and Geometry: Methods and Applications}
}

@article{stapleton2016geometry,
  title={Geometry and stability of tautological bundles on Hilbert schemes of points},
  author={Stapleton, David},
  journal={Algebra \& Number Theory},
  volume={10},
  number={6},
  pages={1173--1190},
  year={2016},
  publisher={Mathematical Sciences Publishers}
}

@article{mori1979projective,
  title={Projective manifolds with ample tangent bundles},
  author={Mori, Shigefumi},
  journal={Annals of Mathematics},
  volume={110},
  number={3},
  pages={593--606},
  year={1979},
  publisher={JSTOR}
}

@article{mehta1987varieties,
  title={Varieties in positive characteristic with trivial tangent bundle},
  author={Mehta, Vikram B and Srinivas, Vasudevan},
  journal={Compositio Mathematica},
  volume={64},
  number={2},
  pages={191--212},
  year={1987}
}

@article{liu2023moment,
  title={On moment map and bigness of tangent bundles of G-varieties},
  author={Liu, Jie},
  journal={Algebra \& Number Theory},
  volume={17},
  number={8},
  pages={1501--1532},
  year={2023},
  publisher={Mathematical Sciences Publishers}
}

@article{tian1992stability,
  title={On stability of the tangent bundles of Fano varieties},
  author={Tian, Gang},
  journal={Internat. J. Math},
  volume={3},
  number={3},
  pages={401--413},
  year={1992}
}

@article{conde2004manifolds,
  title={On manifolds whose tangent bundle is big and 1-ample},
  author={Conde, Luis Eduardo Sol{\'a} and Wi{\'s}niewski, Jaros{\l}aw A},
  journal={Proceedings of the London Mathematical Society},
  volume={89},
  number={2},
  pages={273--290},
  year={2004},
  publisher={Cambridge University Press}
}

@article{bansal2025symmetric,
  title={Symmetric power of higher dimensional varieties},
  author={Bansal, Ashima and Sarkar, Supravat and Vats, Shivam},
  journal={arXiv preprint arXiv:2508.12654},
  year={2025}
}

@article{gupta2015survey,
  title={A survey on Zariski cancellation problem},
  author={Gupta, Neena},
  journal={Indian Journal of Pure and Applied Mathematics},
  volume={46},
  number={6},
  pages={865--877},
  year={2015},
  publisher={Springer}
}

@article {BHPS,
    AUTHOR = {Bhatt, Bhargav and Ho, Wei and Patakfalvi, Zsolt and Schnell,
              Christian},
     TITLE = {Moduli of products of stable varieties},
   JOURNAL = {Compos. Math.},
  FJOURNAL = {Compositio Mathematica},
    VOLUME = {149},
      YEAR = {2013},
    NUMBER = {12},
     PAGES = {2036--2070},
      
}

@article{ganong1985tangent,
  title={The tangent bundle of a ruled surface},
  author={Ganong, Richard and Russell, Peter},
  journal={Mathematische Annalen},
  volume={271},
  number={4},
  pages={527--548},
  year={1985},
  publisher={Springer}
}

@article{mukherjee2025diagonal,
  title={Diagonal property and weak point property of higher rank divisors and certain Hilbert schemes},
  author={Mukherjee, Arijit and Nagaraj, DS},
  journal={Bulletin des Sciences Math{\'e}matiques},
  volume={198},
  pages={103541},
  year={2025},
  publisher={Elsevier}
}

@article{atiyah1957vector,
  title={Vector bundles over an elliptic curve},
  author={Atiyah, Michael Francis},
  journal={Proceedings of the London Mathematical Society},
  volume={3},
  number={1},
  pages={414--452},
  year={1957},
  publisher={Oxford University Press}
}

@article{poonen2002grothendieck,
  title={The Grothendieck ring of varieties is not a domain},
  author={Poonen, Bjorn},
  journal={Mathematical Research Letters},
  volume={9},
  number={4},
  pages={493--497},
  year={2002},
  publisher={International Press of Boston, Inc. Somerville, MA 02143, USA}
}

@article{atiyah1957complex,
  title={Complex analytic connections in fibre bundles},
  author={Atiyah, Michael Francis},
  journal={Transactions of the American Mathematical Society},
  volume={85},
  number={1},
  pages={181--207},
  year={1957},
  publisher={JSTOR}
}

@article{catanese2022manifolds,
  title={Manifolds with trivial Chern classes I: hyperelliptic manifolds and a question by Severi},
  author={Catanese, Fabrizio},
  journal={arXiv preprint arXiv:2206.02646},
  year={2022}
}
\vspace{40pt}
\begin{flushleft}
{\scshape Department of Mathematics, Fine Hall, Princeton University, Princeton, NJ 700108, USA}.

{\fontfamily{cmtt}\selectfont
\textit{Email address: ss6663@princeton.edu} }
\end{flushleft}

\end{document}